\documentclass{amsart}

\usepackage[inactive]{srcltx} 
\usepackage{enumerate, amsmath, amsthm, amscd, amsfonts, amssymb}
\usepackage{hyperref}

 \newtheorem{thm}{Theorem}[section]
 \newtheorem{cor}[thm]{Corollary}
 \newtheorem{lem}[thm]{Lemma}
 \newtheorem{prop}[thm]{Proposition}
 \theoremstyle{definition}
 
 \newtheorem{rem}[thm]{Remark}
 \newtheorem{exm}[thm]{Example}
 \theoremstyle{remark}
 
\newcommand{\A}{\mathcal{A}}      
\newcommand{\U}{\mathcal{U}}      
\newcommand{\C}{\mathfrak{S}}
\newcommand{\X}{\mathcal{X}}
\newcommand{\Y}{\mathcal{Y}}

\newcommand{\M}{\mathcal{M}}
\newcommand{\lau}{\A\times_{\theta}\U}
\DeclareMathOperator{\ann}{ann}

\hypersetup{citecolor=green}

\begin{document}

\nocite{*}

\title[Automatic continuity of Some linear mappings from certain products...]{Automatic continuity of Some linear mappings from certain products of Banach algebras}
\author{ H. Farhadi, E. Ghaderi, H. Ghahramani}
\thanks{{\scriptsize
\hskip -0.4 true cm \emph{MSC(2010)}: 46H40; 13N15; 42A45; 46H25.
\newline \emph{Keywords}: automatic continuity, derivation, Lau product, multiplier.\\}}

\address{Department of
Mathematics, University of Kurdistan, P. O. Box 416, Sanandaj,
Iran.}
\email{h.farhadi@sci.uok.ac.ir}
\email{eg.ghaderi@uok.ac.ir}
\email{h.ghahramani@uok.ac.ir; hoger.ghahramani@yahoo.com}

\maketitle
\begin{abstract}
Let $\A$ and $\U$ be Banach algebras and $\theta$ be a nonzero character on $\A$. Then the \textit{Lau product Banach algebra} $\A\times_{\theta}\U$ associated with the Banach algebras $\A$ and $\U$ is  the $l^1$-direct sum $\A\oplus\U$ equipped with the algebra multiplication  $(a,u)(a',u')=(ab,\theta(a)u'+\theta(a')u+uu')\, (a,a'\in\A,u,u'\in\U)$ and $l^1$-norm. In this paper we shall investigate the derivations and multipliers from this Banach algebras and study the automatic continuity of these mappings. We also study continuity of the derivations for some  special cases of  Banach algebra $\U$ and Banach $\lau$-bimodule $\X$ and establish various results on the continuity of derivations and give some examples. 
\end{abstract}
\section{Introduction}
Let $\A$ be a Banach algebra (over the complex field $\mathbb{C}$), and $\X$ be a Banach $\A$-bimodule. A linear mapping $D:\A\rightarrow\X$ is called a \textit{derivation} if $D (ab)=aD(b)+D(a)b$ for all $a,b\in\A$. For any $x\in\X$, the mapping $ad_x:\A\rightarrow\X$ given by $ad_x(a)=ax-xa$ is a continuous derivation called \textit{inner}. Let $D:\A\rightarrow\A$ be a derivation. Then by a \textit{generalized $\delta$-derivation} we mean a linear mapping $\delta:\X\rightarrow\X$  satisfying $\delta (xa)=\delta(x)a+xD(a)$ for all $a\in\A$ and $x\in\X$.
\par 

The problem of continuity of linear mapping between two Banach algebras (or Banach spaces, in general) lies in the theory of automatic continuity which is an important topic in mathematical analysis. This thoery has been an active field of research during the last fifty years. The automatic continuity of the derivations on different structures in mathematics has attracted the attention of the researchers and, specifically, has been mainly developed in the context of Banach algebras and studied extensively (see for example, \cite{ram}, \cite{loy3}, \cite{loy2}, \cite{pera}, \cite{rund}, \cite{vil4}). The continuity of derivations from a Banach algebra into a Banach bimodule arises in a number of situations. In particular, it arises in cohomology theory of Banach algebras and also in the theory of extensions of Banach algebras. The reader is referred to \cite{da} which is a detailed source in this context. Here we mention the most important established results concerning continuity of derivations. A celebrated theorem due to Johnson and Sinclair \cite{john1} states that every derivation on a semisimple Banach algebra is continuous. Ringrose \cite{rin} showed that every derivation from a $C^{*}$-algebra $\A$ into a Banach $\A$-bimodule $\X$ is continuous. The automatic continuity of module derivations from some special classes of Banach algebras is studied by several authors; for instance, Bade and Curtis \cite{ba3} studied the structure and continuity of the Banach algebra $C^{n}(I)$ of $n$ times continuously differentiable functions on an interval $I$ into $C^{n}(I)$-Banach bimodules. By determining the value of the derivation on certain semigroups, in his paper \cite{gr}, the author describes bounded derivations from commutative Banach algebras into commutative Banach bimodules. In \cite{chr}, Christensen proved that every derivation from a nest algebra on the Hilbert space $\mathcal{H}$ into $\mathbb{B}(\mathcal{H})$ is continuous. Additionally, some results on automatic continuity of the derivations on prime Banach algebras have been established by Villena in \cite{vi1} and \cite{vi2}.\\
\par
Recall that for a Banach  algebra $\A$, a linear mapping $T:\A\rightarrow\A$ is called a \textit{multiplier} on $\A$ whenever $aT(b)=T(a)b$ for all $a,b\in\A$.  $\A$ is said to be \textit{faithful}, if for any $x\in\A$, $Ax=\{0\}=x\A$ implies that $x=0.$ It is well-known and easy to show that if $\A$ is faithful, then every multiplier on $\A$ is continuous. The notion of multiplier was originally introduced by Helgason \cite{hel}, as a bounded continuous function $g$ defined on the regular maximal ideal space $\Delta(\A)$ such that $g(\hat{\A})\subseteq \hat{\A}$ where $\hat{\A}$ denotes the Gelfand representation of $\A$, and has been developed by Wang  \cite{wan} and Birtal \cite{bir}.  The thoery of multipliers has an important applications in many areas of harmonic analysis and as well as in optimization theory, differential equations, probability, mathematical finance and economics. A good reference for this thoery is the monograph of Larsen \cite{lar} (see also Laursen and Neumann \cite{laur}). \\
\par 
Let $\A$ and $\U$ be Banach algebras and $\theta:\A\rightarrow\mathbb{C}$ be a non-trivial character on $\A$. We equip the space $\A\times\U$ with the usual $\mathbb{C}$-module structure. Then the multiplication  
\[(a,u)(a',u')=(aa',\theta(a)u'+\theta(a')u+uu')\quad(a,a'\in\A, u, u'\in\U)\]
with the norm 
\[||(a,u)||=||a||+||u||,\]
turn $\A\times\U$ into a Banach algebra called Lau  products of Banach algebras, denoted by $\A\times_{\theta}\U$. 

Note that  we identify $\A \times \{0\}$ with $\A$, and $\{0\} \times \U $ with $\U$, thus $\A$ is a closed subalgebra of  $\lau$ while $\U$ is a closed ideal of it, and
\[ \frac{\A\times_{\theta}\U}{\U} \cong \A \quad (isometrically \, \, isomorphism). \]
Indeed $\lau$ is equal to direct sum of $\A$ and $\U$ as Banach spaces.
\par

The Lau product was first introduced by T. Lau  \cite{lau} for special classes of Banach algebras that are predual of a von Neumann algebra and for which the identity of the dual is a multiplicative linear functional. Monfared \cite{mn}, has studied and verified some structural and topological properties of this special product. The reader can find more information in \cite{mn} and references therein.
\par 
Let $\X$ be a Banach $(\lau)$-bimodule. Then $\X$ is  a Banach $\A$-bimodule by defining module operations in a natural fashion;
\[a.x=(a,0)x\quad,\quad x.a=x(a,0)\quad (a\in\A,x\in\X).\]
Similarly $\X$ turns into a Banach $\U$-bimodule with the module actions given by
\[u.x=(0,u)x\quad,\quad x.u=x(0,u)\quad (u\in\U,x\in\X).\]
\par
A key notion to study the automatic continuity of a linear mapping $T:\mathcal{X}\rightarrow \mathcal{Y}$ between two Banach spaces $\mathcal{X}$ and $\mathcal{Y}$ is the separating space which is defined as
\[\mathfrak{S}(T):=\{y\in \mathcal{Y}\, \mid \, \text{there is}\, \, \{x_n\}\subseteq \mathcal{X}\,\, \text{with}\, \, x_n\rightarrow 0 , \, T(x_n)\rightarrow y\} .\]
Note that by the closed graph theorem, $T$ is continuous if and only if $\mathfrak{S}(T)=\{0\}$. For a thorough discussion of the separating space one can refer to \cite{da} and \cite{sinc}. 

For a derivation $D:\A\rightarrow\X $ the two-sided continuity ideal is defined to be 
\[\mathcal{I}(D)=\{a\in\A\, :\, a\C(D)=\C(D)a=0\}.\]
Note that a derivation need not be continuous on $ \mathcal{I}(D) $. But rather it is bounded as a bilinear form. However, if $\mathcal{I}(D)$ has a bounded approximate identity, then the restriction of $D$ to its continuity ideal $\mathcal{I}(D)$ is continuous. 
\par 
Let $\A$ be a Banach algebra and $\X$ and $\mathcal{Y}$ be Banach $\A$-bimodules. By $Z(\A)$, we mean the center of $\A$ while for $\mathcal{S}\subseteq\mathcal{X}$ we have
$Z_{\mathcal{S}}(\A)=\{s\in\mathcal{S} : s\A=\A s\}$. Also, the annihilator of $\A$ over $\mathcal{S}$, denoted by $\ann_{\mathcal{S}}\A$ is defined as 
\[ann_{\mathcal{S}}\A:=\{s\in\mathcal{S}\, : \, s\A=\A s=\{0\}\}.\]
Similarly for a subset $ \mathcal{D}\subseteq\A $ we write, 
\[ann_{\mathcal{X}}\mathcal{D}:=\{x\in\mathcal{X}\, : \, x\mathcal{D}=\mathcal{D}x=\{0\}\}.\]
A linear mapping $\phi :\X\rightarrow \mathcal{Y}$ is said to be a \textit{left (resp. right) $\A$-module homomorphism} if $\phi (ax)=a\phi (x)$(resp. $\phi (xa)=\phi (x)a$ ) for all $a\in \A$ and $x\in\X$. $\phi$ is called \textit{$\A$-module homomorphism}, if it is both left and right $\A$-module homomorphism. 
\par 
This paper is organized as follows. In section $2$ we consider derivations of Lau products of Banach algebras and determine the general structure of them and by studying the properties of the appearing maps, we obtain conditions under which these mappings are automatically continuous and establish various results in this context. Since inner derivations form an important class of automatically continuous derivations, some of the results are devoted to investigating the inner-ness of the derivations. We also consider some special cases and study continuity of the derivations and give some examples. In section $3$ we turn our attention to the multipliers of Lau products and obtain some results concerning  automatic continuity of these mappings.  

\section{The Derivations}
Our aim in this section is to study the derivations from $\lau$ and investigate automatic continuity of them.  Also, we obtain several results on the continuity of the derivations foe some special cases.\\  Throughout, $\A,\U$ are Banach algebras, $\theta$ is a nonzero character on $\A$, $ \A\times_{\theta}\U$ denotes the corresponding Lau product  and $\X$ is a Banach $(\A\times_{\theta}\U)$-bimodule. In the following theorem we characterize the derivations  from Lau products and determine the general structure of them. 
\begin{thm}\label{Asli}
Let $D:\A\times_{\theta}\U\rightarrow\X$ be a linear mapping. Then the following statements are equivalent.
\begin{enumerate}[(i)]
\item 
$D$ is a derivation.
\item
There are linear mappings $\delta_1:\A\rightarrow\X$ and $\delta_2:\U\rightarrow\X$ with 
\[D((a,u))=\delta_1(a)+\delta_2(u)\quad(a\in\A,u\in\U)\]
such that $\delta_1$ and $\delta_2$ are derivations and satisfy the following conditions;
\[a\delta_2(u)+\delta_1(a)u=\theta(a)\delta_2(u)=\delta_2(u)a+u\delta_1(a)\quad(a\in\A,u\in\U)\]
\end{enumerate}
\end{thm} 
  \begin{proof}
 $(i)\implies(ii)$ Suppose that  $D:\A\times_{\theta}\U\rightarrow\X$ is a derivation. Then put $\delta_{1}(a)=D((a,0))$ and $\delta_{2}(u)=D((0,u))$. Thus since $D$ is linear, $D((a,u))=\delta_1 (a)+\delta_2 (u)$. If we apply $D$ on the both sides of  $(a,u)(a',u')=(aa',\theta(a)u'+\theta(a')u+uu')$, then it is easy to see that $\delta_{1},\delta_{2}$ satisfy the given equations.

 $(ii)\implies(i)$ If $D$ is a linear mapping satisfying the conditions given in  $(ii)$, then it is routinely checked that $D$ is a derivation. 
    \end{proof}
In view of the above theorem, for every derivation $D:\lau\rightarrow\X$ we can write $D=\delta_1+\delta_{2}$ where $\delta_{1},\delta_{2}$ are as in the above theorem.\\
Since inner derivations are continuous ones, if we show that a given derivation $D$ is inner this, in turn, implies that $D$ is automatically continuous.  In the following result we characterize inner derivations  from $\lau$. 
    \begin{prop}\label{inner1}
    Let $D:\A\times_{\theta}\U\rightarrow\X$ be a derivation with $D=\delta_{1}+\delta_{2}.$ Then
    \begin{enumerate}[(i)]
    \item
      If $D$ is inner, then $\delta_1$ and $\delta_2$ are inner. 
     \item
     If $\delta_1=ad_{x_{0}}$ and $\delta_2=ad_{y_{0}}$, then $D=ad_{z_{0}}$ for some $z_{0}\in\X$ if and only if $z_{0}-x_{0}\in Z_{\X}(\A)$ and $z_{0}-y_{0}\in Z_{\X}(\U)$.
    \end{enumerate}
    \end{prop}
    \begin{proof}
    \begin{enumerate}[(i)]
    \item
    Since $D$ is inner there exists some $x_0\in\X$ for which $D=ad_{x_{0}}=(a,u)x_0 -x_0(a,u)$ for all $a\in\A$ and $u\in\U$. Substituting $a=0$, we obtain  $\delta_2(u)=ux_0 -x_0u$ for all $u\in\U$. Similarly if  $u=0$, then we have $\delta_1(a)=ax_0-x_0a$ for all $a\in\A$. Thus $\delta_1$ and $\delta_2$ are inner. 
    \item
    Assume that $\delta_1=ad_{x_0}$ and $\delta_2=ad_{y_0}$ and there exists $z_0\in\X$ be such that $D=ad_{z_{0}}$. We show that $z_{0}-y_{0}\in Z_{\X}(\U)$. Note that 
   \[(a,u)z_{0}-z_{0}(a,u)=ax_{0}-x_{0}a+uy_{0}-y_{0}u.\] 
   for all $a\in\A$ and $u\in\U$. Substituting  $a=0$ we have
  $ u(z_{0}-y_{0})=(z_{0}-y_{0})u $
   for all $u\in\U$. Thus $(z_{0}-y_{0})\in  Z_{\X}(\U)$.
   By the same way one can show that $z_{0}-x_{0}\in Z_{\X}(\A)$. The converse is clear and we omit it. 
    \end{enumerate}
    \end{proof}
      It is clear  from the above proposition that if $\delta_{1},\delta_{2}$ are inner derivations induced by the same element $x_0$ (i.e., $\delta_1=ad_{x_0}$ and $\delta_2= ad_{x_0}$), then $D$ is always inner since one can take $z_{0}=x_0$. It may happen, however, that $ \delta_1 $ and $ \delta_2 $ are inner but $ D $ is not. The following example shows this. 
      \begin{exm}
      	Consider that Banach algebra $\A$ of upper triangular $3\times 3$ real matrices with $0$ on the diagonal. So $\A^{3}=0$ but $\A^{2}\neq 0$. Let $a_{0},b_{0}$ be two non-central and central elements of $\A$ respectively. Define $\delta_1,\delta_2 :\A\rightarrow\A$ respectively by $\delta_{1}(a)=aa_0 -a_0a$ and $\delta_{2}(a)=ab_0-b_0 a =0$. Then $\delta_1,\delta_{2}$ are inner but $D:\A\times_{\theta}\A\rightarrow\A$ with $D((a,b))=\delta_{1}(a)+\delta_{2}(b)$ is a derivation which is not inner. To show this, assume towards a contradiction that $D=ad_{c_0}$ for some $c_0\in\A$. Then by Proposition \ref{inner1}, $c_0 -a_0 , c_0 -b_0\in Z(\A)$ and so $a_0\in Z(\A)$, a contradiction. 
      \end{exm}
 It is worth noting that the above example can be extended to a more general case by considering $\A$ to be any non-commutative Banach algebra with $\A^{3}=0$.\\

The next result is a consequence of Theorem \ref{Asli}.
\begin{prop}\label{xtheta}
Let $\A$ and $\U$ be Banach algebras and $\X$ be a Banach $(\lau)$-bimodule. Then
\begin{enumerate}[(i)]
\item
If $\delta:\A\rightarrow\X$ is a derivation with  $\delta(\A)\subseteq \ann_{\X}\U$, then $\delta$ extends to a derivation $\tilde{\delta}:\A\times_{\theta}\U\rightarrow \X$. $\tilde{\delta}$ is continuous if and only if $\delta$ is continuous. Moreover, if $\delta=ad_{x_0}$ for some $x_0\in Z_{\X}(\U)$, then so is $\tilde{\delta}$.  
\item
Every derivation $D:\U\rightarrow\X$ gives rise to a derivation $\tilde{D}:\A\times_{\theta}\U\rightarrow\X$.
 $\tilde{D}$ is continuous if and only if $D$ is continuous. Moreover, if $D=ad_{x_0}$ for some $x_0\in Z_{\X}(\A)$, then so is $\tilde{D}$.  
\end{enumerate}
\end{prop}
\begin{proof}
\begin{enumerate}[(i)]
\item
It is clear that by the following module actions $\X$ turns into a Banach $(\lau)$-bimodule,
\[x.(a,u)=xa,\quad\quad (a,u).x=ax\]
for all $a\in\A, u\in\U$ and $ x\in\X$. 
Define $\tilde{\delta}:\A\times_{\theta}\U\rightarrow \X$ by $\tilde{\delta}((a,u))=\delta(a)$. Then since $\delta(\A)\subseteq \ann_{\X}\U$, by Theorem \ref{Asli}(ii), $\tilde{\delta}$ is a derivation. If $\delta$ is continuous, then so is $\tilde{\delta}$. Now if $\delta=ad_{x_0}$ for some $x_0\in Z_{\X}(\U)$, then
\begin{eqnarray*}
\tilde{\delta}((a,u))&=&\delta(a)\\&=&(a,u)x_0-x_0 (a,u)=ad_{x_0}\quad \quad (a\in\A,u\in\U).
\end{eqnarray*}
\item
Suppose that $D:\U\rightarrow\X$ is a derivation. Then by the module actions given by 
\[x.(a,u)=\theta(a)x+xu,\quad  (a,u).x=\theta(a)x+ux\quad ( (a,u)\in\A\times\U,x\in\X),\]
$\X$ is a Banach $(\lau)$-bimodule. Define $\tilde{D}:\lau\rightarrow\X$ by $\tilde{D}((a,u))=D(u)$ for all $a\in\A,u\in\U$. So
\begin{eqnarray*}
\tilde{D}((a,u))(a',u'))&=&\theta(a') D(u)+D(u)u'+\theta(a) D(u')+uD(u')\\&=&D(u)(a',u')+(a,u)D(u')\\&=&\tilde{D}((a,u))(a',u')+(a,u)\tilde{D}((a',u')) \quad \quad (a,a'\in\A, u,u'\in\U).
\end{eqnarray*}
Thus $\tilde{D}$ is a derivation. Moreover, if $D=ad_{x_{0}}$ for some $x_{0}\in Z_{\X}(\A)$,
\begin{eqnarray*}
\tilde{D}((a,u))&=&D(u)\\&=&  (a,u)x_0 -x_0 (a,u)=ad_{x_0}\quad\quad (a\in\A,u\in\U).
\end{eqnarray*}
\end{enumerate}
\end{proof}
Note that in the preceding proposition the inner-ness of $\tilde{\delta}$ (resp. $D$) implies so is $\delta$ (resp. $\delta$). This follows directly from Proposition \ref{inner1}-(i).
\par
In this part of the paper, we investigate the relation between the separating spaces of $\delta_{1},\delta_{2}$ and use the results to study the automatic continuity of the derivations. 
\begin{thm}\label{joda}
Let $D:\A\times_{\theta}\U\rightarrow\X$ be a derivation and $\delta_{1},\delta_{2}$ be as in Theorem \ref{Asli}. Then
\begin{enumerate}[(i)]
\item
$\C(\delta_1)$ is an $\A$-subbimodule of $\X$ and $\C(\delta_1)\subseteq \ann_{\X}\U$. 
\item
$\C(\delta_2)$ is a symmetric $\A$-subbimodule of $\X$ and $\C(\delta_2)\subseteq Z_{\X}(\A)$. Moreover $\C(\delta_2)$ is a $\U$-subbimodule of $\X$, too.
\end{enumerate}
\end{thm}
\begin{proof}
\begin{enumerate}[(i)]
\item
We only prove the given inclusion. Let $x_0\in\C(\delta_1)$. Then there exists some sequence $a_n$ in $\A$ such that $a_n\to 0$ and $\delta_1(a_n)\to x_0$. So we have 
\[\delta_2(u)a_n +u\delta_1(a_n)=\theta(a_n)\delta_2(u)\]
for all $u\in\U$. Letting $n$ tend to infinity, we get $ux_0=0$ and similarly $x_{0}u=0$ for all $u\in\U$. Hence $\C(\delta_1)\subseteq \ann_{\X}\U$.
\item
Let $y_0\in\C(\delta_2)$. By Theorem \ref{Asli},
\[\delta_2(u_n)a+u_{n}\delta_1(a)=a\delta_2(u_n)+\delta_1(a)u_{n}\]
for some sequence $u_n$ in $\U$ with $u_n\to 0$ and all $a\in\A$. By tending $n$ to infinity, we obtain
$y_{0}a=ay_0$ for all $a\in\A$. Thus $\C(\delta_2)\subseteq Z_{\X}(\A)$.
\end{enumerate}
\end{proof}
\begin{thm}\label{delta1continuous}
Let $D:\A\times_{\theta}\U\rightarrow\X$ be a derivation and $\delta_{1},\delta_{2}$ be as in Theorem \ref{Asli}. Then
\begin{enumerate}[(i)]
\item
If $\ann_{\X}\U=\{0\}$ or $\U$ has a left (or right) bounded approximate identity for $\X$, then $\delta_1$ is  continuous.
\item
 Suppose that $Z_{\X}(\A)=\{0\}$. Then $\delta_2:\U\rightarrow\X$ is a continuous derivation. In addition if  $\ann_{\X}\U=\{0\}$, then any derivation $D:\A\times_{\theta}\U\rightarrow\X$ is continuous.
\end{enumerate}
\end{thm}
\begin{proof}
\begin{enumerate}[(i)]
\item
If $\ann_{\X}\U=\{0\}$, by Theorem \ref{joda} it follows that $\C(\delta_1)=\{0\}$. Thus $\delta_{1}$ is continuous. 
\item
If $Z_{\X}(\A)=\{0\}$, the second part of Theorem \ref{joda} implies $\C(\delta_2)=\{0\}$. Hence $\delta_2$ is continuous. 
\end{enumerate}
\end{proof}

If $\X$ and $\Y$ are Banach $\A,\U$-bimodule respectively, it can be seen that the module actions 
\[(a,u).x=ax\quad ,\quad x.(a,u)=xa\]
and
\[(a,u).y=\theta(a)y+uy\quad , \quad y.(a,u)=\theta(a)y+yu\quad(a\in\A,u\in\U,x\in\X,y\in\Y).\]
turn $\X$ into a Banach $(\lau)$-bimodule. 
Now consider $\tilde{\M}=\X\times \Y$. $\tilde{\M}$ becomes a Banach $(\lau)$-bimodule with the module actions given by
\begin{eqnarray*}
	(a,u).(x,y)&=&(ax,\theta(a)y+uy)\\ (x,y).(a,u)&=&(xa,\theta(a)y+yu)\quad(a\in\A,u\in\U,x\in\X,y\in\Y).
\end{eqnarray*}

\begin{thm}
	Let $\X,\Y$ be Banach $\A,\U$-bimodules respectively and $\tilde{\M}$ defined as above. Then $D:\lau\rightarrow\tilde{\M}$ is derivation if and only if  
	\[D((a,u))=d_1(a)+d_2(u)\quad(a\in\A,u\in\U)\]
	where $d_1:\A\rightarrow\X$ and $d_2:\U\rightarrow\Y$ are derivations. Moreover, $D$ is inner if and only if $d_{1},d_{2}$ are inner in such a way if $D=ad_{z}(z\in\tilde{\M})$ with $z=(x,y)\in\tilde{\M}$, then $d_{1}=ad_{x}$ and $d_2=ad_{y}$ and vice versa.
\end{thm}
\begin{proof}
	It can be routinely checked that $D$ is a derivation if and only if $d_{1}, d_{2}$ are derivations. For the second part, suppose that $D=ad_{z_{0}}$ for some $z=(x,y)\in\tilde{\M}$. Then we have
	\begin{eqnarray*}
		D((a,u))&=&(a,u)(x,y)-(x,y)(a,u)\\&=&(ax-xa,uy-yu)
	\end{eqnarray*}
	for all $(a,u)\in\lau$ and $(x,y)\in\tilde{\M}$. So $\delta_1(a)=ad_{x}$ and $\delta_2(u)=ad_{y}$ are inner. The other direction can be done similarly so we omit its proof.
\end{proof}
\subsection{Special cases}
Let $\A,\U$ and $\X$ be as in the previous section. In this subsection we shall study automatic continuity of the derivations from Lau Banach algebras for some special cases of $\X$ and $\U$ and establish various results in this context. 
\subsubsection*{$ \X $ \textbf{is a simple Banach} $ (\lau) $-\textbf{bimodule}}
In this part we assume that $\X$ is a simple Banach $(\lau)$-bimodule and  obtain some  results as follows.
\begin{thm}\label{simp delta1}
Suppose that $\X$ is a simple Banach $(\lau)$-bimodule and $D:\lau\rightarrow\X$ with $D=\delta_{1}+\delta_{2}$ is any derivation. Then either $\delta_1$ is continuous or $\ann_{\X}\U=\X$. 
\end{thm}
\begin{proof}
Since $\delta_1$ is a derivation, by Theorem \ref{joda}, $\C(\delta_1)$ is an $\A$-subbimodule of $\X$. Since $\X$ is simple, we have either $\C(\delta_1)=\{0\}$, from which we conclude that $\delta_1$ is continuous, or $\C(\delta_1)=\X$. If $\C(\delta_1)=\X$, the same Theorem implies that $\ann_{\X}\U=\X$.
\end{proof}
\begin{thm}\label{simp delta2}
Let $\X$ be a simple Banach $(\lau)$-bimodule and $D:\lau\rightarrow\X$ be a derivation with $D((a,u))=\delta_1(a)+\delta_2(u)$. Then $\delta_2$ is continuous or $\X$ is a symmetric Banach $\A$-bimodule.
\end{thm}
\begin{proof}
	$\C(\delta_2)$ is an $\U$-subbimodule of $\X$. Since $\X$ is simple, so either $\C(\delta_2)=\{0\}$ or $\C(\delta_2)=\X$. The former implies that $\delta_2$ is continuous. If $\C(\delta_2)=\X$, then since $\C(\delta_2)\subseteq Z_{\X}(\A)$, then $Z_{\X}(\A)=\X$  or $\A\X=\X\A$.
\end{proof}
By the above theorem we conclude that  if there are two non-commuting elements $a_0 \in\A, x_0\in\X$, then $\delta_2$ is continuous. 
\par
As a consequence of Theorems \ref{simp delta1} and \ref{simp delta2}, we state the following result.
\begin{cor}
Suppose $\X$ is a simple Banach $(\lau)$-bimodule and $D:\lau\rightarrow\X$ is a derivation with $D=\delta_1+\delta_2$. Then $D$ is continuous if either of the following conditions holds. 
\begin{enumerate}[(i)]
\item
$\ann_{\X}\U\neq \X$ and $Z_{\X}(\A)\neq \X$.
\item
$\ann_{\X}\U=\{0\}$ and $Z_{\X}(\A)\neq \X$.
\end{enumerate}
\end{cor} 
\begin{proof}
\begin{enumerate}
\item[(i)]
It is clear by Theorems \ref{simp delta1} and \ref{simp delta2}  that $D$ is continuous.
\item[(ii)]
Follows directly from Theorem \ref{simp delta2} and Corollary \ref{delta1continuous}.
\end{enumerate}
\end{proof}
\subsubsection*{\textbf{The case} $ \U=\A $}
In this part we study derivations $D:\A\times_{\theta}\A\rightarrow\X$ where $\A$ is a Banach algebra and $\X$ a Banach $(\A\times_{\theta}\A)$-bimodule. 

 According to Theorem \ref{delta1continuous}, if $Z_{\X}(\A)=\{0\}$, then  since $\ann_{\X}\A\subseteq Z_{\X}(\A)$, every derivation $D:\A\times_{\theta}\A\rightarrow\X$ is continuous. We give the following example satisfying these conditions.
\begin{exm}
	\begin{enumerate}[(i)]
		\item 
		Let $\A=\mathcal{K}(\X)$, the compact operators on an infinite dimensional Banach space $ \X $. Then since $ Z(\A)=\{0\}$, any derivation $D:\mathcal{K}(\X)\times_{\theta}\mathcal{K}(\X)\rightarrow\mathcal{K}(\X)$ is automatically continuous. 
		\item 
		Suppose that $\mathcal{B}$ is a rectangular band and $\A=\ell ^{1}(\mathcal{B})$. Then every derivation $ D:\ell ^{1}(\mathcal{B})\times_{\theta}\ell ^{1}(\mathcal{B})\rightarrow \ell ^{1}(\mathcal{B}) $ is continuous.
	\end{enumerate}
\end{exm}
\begin{lem}\label{continuity ideal}
	Let $ D:\A\times_{\theta}\A\rightarrow\X $ be a derivation with $ D=\delta_1 + \delta_2 $. Then $ \mathcal{I}(\delta_1)=\A $.
\end{lem}
\begin{proof}
	The result immediately follows from the definition of a continuity ideal   with Theorem \ref{joda}-(i). 
\end{proof}
\begin{cor}
	Let $ D:\A\times_{\theta}\A\rightarrow\X $ be a derivation with $D=\delta_1+\delta_2 $. Then for any $ a\in\A $,  the linear mapping $d_{a}:\A\rightarrow \X $ given by $ D_{a}(b)=a\delta_{1}(b) $ is a continuous derivation.  
\end{cor}
\begin{proof}
	By Theorem 3.2 of \cite{ba4} 	the continuity ideal $ \mathcal{I}(D) $ of every derivation $ D $ coincides with $ \{a\in\A\,\mid\, D_{a}\,  \text{is continuous} \}$ where $ D_{a}(b)=a\delta_{1}(b) $. Now the result is clear by Lemma \ref{continuity ideal}.
\end{proof}

	Let $\A$ and $\U$ be two Banach algebras. Then it is easy to see that by module actions
	\[a.u=\theta(a)u,\quad u.a=\theta(a)u,\]
	for all $a\in\A$ and $u\in\U$, $\U$ becomes a Banach $\A$-bimodule. Therefore using Theorem 
\ref{Asli}, $\delta_2$ is a generalized $\delta_1$-derivation that satisfies 
\[\delta_2(ua)=\delta_2(u)a+u\delta_1(a)\] 
for all $a\in\A$ and $u\in\U$.
The generalized derivation $\delta_2$ appeared naturally in the decomposition of derivations $D:\A\times_{\theta}\U\rightarrow\X$. The next theorem connects the continuity of $\delta_{1}, \delta_{2}$ for the case $\U=\A$ to that of $\A$-bimodule homomorphisms.

\begin{thm}
Suppose that $D:\A\times_{\theta}\A\rightarrow\X$ is a derivation with $D=\delta_1 +\delta_2.$ Then $\delta_2$ is a generalized $\delta_1$-derivation if and only if $\delta_2-\delta_1$ is a right $\A$-bimodule homomorphism.  
\end{thm}
\begin{proof}
First suppose that  $\delta_2$ ie a generalized $\delta_1$-derivation. Then we have 
\begin{eqnarray*}
(\delta_2-\delta_1)(ab)&=& \delta_{2}(a)b+a +a\delta_{1}(b)-\delta_1(a)b-a\delta_{1}(b)\\
&=& (\delta_2-\delta_1)(a)\, b\\
\end{eqnarray*} 
for all $a,b\in\A$. Conversely, if $ \delta_2-\delta_1 $ is a right $\A$-bimodule homomorphism, then by an easy calculation it can be seen that $\delta_2$ is a generalized $\delta_1$-derivation.
\end{proof}
\begin{rem}
	Note that a direct application of the Cohen factorization theorem shows that if $\A$ has a bounded approximate identity for $\X$, then every $\A$-bimodule homomorphism $\phi :\A\rightarrow \X$ is continuous. Indeed,  let $(a_n)\subseteq\A$ be a sequence with $a_n\to 0$. Then by Cohen factorization theorem there exist a sequence $(b_n)$ in $\A$ converging to zero and some $c\in\A$ such that $a_n=cb_n$, so 
	$\phi(a_n)=\phi(c)b_n\to 0$. Thus $\phi$ is continuous.
\end{rem}
\begin{cor}
Let $\A$ be a Banach algebra which has a bounded approximate identity and  $D:\A\times_{\theta}\A\rightarrow\X$ be a derivation with $D=\delta_1+\delta_2$. Then $\delta_1$ is continuous if and only if $\delta_2$ is continuous.
\end{cor}
\begin{cor}
Let $D:\A\times_{\theta}\A\rightarrow\X$ be a derivation with $D=\delta_1+\delta_2$. Then  $(\delta_2-\delta_1)(\A)\subseteq \ann_{\X}\A$.
\end{cor}
\begin{proof}
We know $\delta_2$ satisfies 
\[\delta_2(ab)=\delta_2(a)b+a\delta_2(b)\quad \text{and}\quad
\delta_2(ab)=\delta_2(a)b+a\delta_1(b)\]
for all $a,b\in\A$. So $\delta_2(a)b=\delta_1(a)b\,(a,b\in\A)$. Thus $(\delta_2-\delta_1)(\A)\subseteq \ann_{\X}\A$.
\end{proof}
In the case where $\ann_{\X}\A=\{0\}$, $\delta_{1},\delta_{2}$ agree on $\A$. For instance if $\A$ has a bounded approximate identity for $\X$, then $\ann_{\X}\A=\{0\}$. 
\begin{cor}
Let $D:\A\times_{\theta}\A\rightarrow\X$ be a derivation with $D=\delta_1+\delta_2$ If $\ann_{\X}\A=\{0\}$, then $\delta_2=\delta_1$. In this case any derivation $D:\A\times_{\theta}\A\rightarrow\X$ can be written as $D((a,b))=\delta_1(a)+\delta_1(b)=\delta_1(a+b)$.
\end{cor}

\subsubsection*{The case $\X=\U$}~\\
As we noted before, $\U$ is an ideal in $\lau$. So  $\U$ is a Banach $(\lau)-$bimodule as well. The following proposition is a special case of Theorem $\ref{Asli}$.
\begin{prop}\label{Asliu}
Let $\A$ and $ \U$ be two Banach algebras. Then $D:\lau\rightarrow\U$ is a derivation if and only if $ D=\delta_1 +\delta_2 $ such that 
$\delta_1:\A\rightarrow\U$ and $\delta_2:\U\rightarrow\U$ are derivations and 
\[\theta(a)\delta_2(u)=\delta_1(a)u+a\delta_2(u)= u\delta_1(a)+\delta_2(u)a\quad (a\in\A,u\in\U).\]
\end{prop}
In the next theorem we state some results similar to those of Theorem \ref{joda} and Theorem \ref{delta1continuous}. Using the results we study the continuity of the derivations $D:\lau\rightarrow\U$. 
\begin{thm}
Let $D:\lau\rightarrow\U$ be a derivation with $D=\delta_1+\delta_2$. Then 
\begin{enumerate}[(i)]
\item
$\C(\delta_1)$ is an $\A$-subbimodule of $\U$ and $\C(\delta_2)$ is an ideal in $\U$. In particular, $\A\C(\delta_2)=\C(\delta_2)A=\theta(\A)\C(\delta_2).$ 
\item
$\U$ annihilates $\C(\delta_1)$; that is, $\U\C(\delta_1)=\C(\delta_1)\U=\{0\}$.
\end{enumerate}
\end{thm}
\begin{cor}
Suppose that $D:\lau\rightarrow\U$ is a derivation with $D=\delta_1+\delta_2$ and  $\ann_{\U}\U =\{0\}$, then $\delta_1$ is continuous. In this case $D$ is continuous if and only if $\delta_2$ is continuous.
\end{cor}
\begin{proof}
By the preceding theorem, $\C(\delta_1)\U=\U\C(\delta_1)=\{0\}$. So $\C(\delta_1)=\{0\}$ and thus $\delta_1$ is continuous. 
\end{proof}
Note that any Banach algebra $\U$ possessing a bounded approximate identity satisfies the hypothesis of the above corollary;  since in this case $\ann_{\U}\U =\{0\}$. For the case $\U$ is semisimple, a well-known result of Johnson implies the continuity of $\delta_2$.
\begin{prop}
Let $\A$ and $\U$ be Banach algebras such that $\U$ is semisimple. Then every derivation $D:\lau\rightarrow\U$ is continuous.
\end{prop}
\begin{proof}
Since $\U$ is semisimple, by the Johnson theorem $\delta_2:\U\rightarrow\U$ is continuous. On the other hand, $\ann_{\U}\U=\{0\}$, thus by Proposition \ref{delta1continuous}-(i), $\delta_1$ is continuous as well. Therefore every derivation $D:\lau\rightarrow\U$ is continuous.
\end{proof}
All $C^{*}$-algebras, semigroup algebras, measure algebras and unital simple algebras are semisimple Banach algebras with a bounded approximate identity. Thus these classes of Banach algebras satisfy the hypothesis of the above proposition. Consequently, we have the following result.
\begin{cor}
Suppose that $\A$ is a Banach algebra and $\U$ is a $C^{*}$-algebra. Then every derivation $D:\lau\rightarrow\U$ is continuous.  
\end{cor}
The last case which will be discussed is the case where Banach $( \lau )$-bimodule $ \X $ is $ \lau $ itself.
\subsubsection*{The case $\X =\lau$}
In this part our aim is to study  derivations $D:\lau\rightarrow\lau$ and investigate the automatic continuity of them. Note that as in Proposition \ref{xtheta}, $\U$ can be regarded as a Banach $\A$-bimodule. In the following theorem we determine the structure of the derivations on $\lau$.  
\begin{thm}\label{lau-der}
	Let $D:\lau\rightarrow \lau$ be a mapping. Then the following statements are equivalent.
	\begin{enumerate}
		\item[(i)] $D$ is a derivation.
		\item[(ii)] 
		\[D((a,u))=(\delta_1 (a)+\tau_1 (u),\delta_2 (a)+\tau_2 (u))\quad\quad \] for all $a\in \A,u\in \U$ where 
		\begin{enumerate}[(a)]
			\item
			$\delta_1 :\A\rightarrow\A, \delta_2 :\A\rightarrow\U$ are derivations such that 
			\[\theta (\delta_1 (a))u+\delta_2(a)u=0\quad \text{and}\quad \theta (\delta_1 (a))u+u\delta_2(a)=0\quad (a\in\A,u\in\U).\]
			\item
			$\tau_1:\U\rightarrow \A$ is an $\A$-bimodule homomorphism such that $\tau_1(uu')=0\,(u,u'\in\U)$.
			\item
			$\tau_2:\U\rightarrow \U$ is a linear mapping satisfying 
			\[\tau_2(uu')=\theta (\tau_1(u))u'+\theta (\tau_1(u'))u+u\tau_2(u')+\tau_2(u)u'\quad \quad (u,u'\in\U).\]
			Also $D$ is inner if and only if $\tau_1 =0, \delta_2 =0$, $\delta_1$ and $\tau_2$ are inner.
		\end{enumerate}
	\end{enumerate}
\end{thm}
By the above theorem, for a derivation $D$ on $\lau$ we have 
\[\delta_2 (\A)\subseteq Z(\U), \quad \theta (a)\tau_1 (u) =a\tau_1(u)=\tau_1(u)a,\]
and so $\tau_1(\U)\subseteq Z(\A)$. Also $u\tau_1 (u')+\tau_1 (u)u' =0$ for all $u,u'\in \U$ if and only if $\tau_1 (\U)\subseteq \ker \theta$. Additionally, $\delta_1 (\A)\subseteq ann_{\A}\U =\ker\theta$ if and only if $\delta_2 (\A)\subseteq ann_{\U}\U$.
\begin{cor}\label{tak}
	Suppose that $\delta_1:\A\rightarrow \A$,  $\delta_2:\A\rightarrow \U$, $\tau_1 :\U\rightarrow \A$ and $\tau_2 :\U\rightarrow \U$ are linear mappings. Then 
	\begin{enumerate}
		\item[(i)]
		$D:\lau\rightarrow \lau $ defined by $D((a,u))=(\delta_1(a),0)$ is a derivation if and only if $\delta_1$ is a derivation. Also, $D$ is inner if and only if $\delta_1$ is inner. 
		\item[(ii)]
		$D:\lau\rightarrow\lau $ with $D((a,u))=(0,\delta_2(a))$ is a derivation if and only if $\delta_2$ is a derivation and $\delta_2 (A)\subseteq \ann_{\U}\U$. Moreover, if $\delta_2 =ad_{u_0}$ is inner and $u_0\in Z(\U)$, then $D$ is inner.
		\item[(iii)]
		$D:\lau\rightarrow \lau $ with $D((a,u))=(\tau_1(u),0)$ is a derivation if and only if $\tau_1(uu')=0, u\tau_1(u')+\tau_1(u)u'=0\, (u.u'\in\U)$. In this case $D$ is inner if and only if $\tau_1 =0$ and $Z(\A)\neq \emptyset$.
		\item[(iv)] 
		$D:\lau\rightarrow \lau $ by $D((a,u))=(0,\tau_2(u))$ is  (inner)derivation if and only if $\tau_2$ is  (inner) derivation.
	\end{enumerate}
\end{cor}
If $\A$ and $\U$ are Banach algebras such that $\A$ is commutative, then by Thomas' theorem \cite{tho}, for every derivation $D$ on $\lau$, $\delta_1(\A)\subseteq rad(\A)\subseteq\ker\theta=\ann_{\A}\U$ and  $\delta_2(\A)\subseteq\ann_{\U}\U$(  $\delta_{1},\delta_{2}$  are as in Theorem \ref{lau-der}). Also, in this case $D=D_1+D_2+D_{3}$ where $D_1((a,u))=(\delta_1(a),0)$, $D_{2}((a,u))=(0,\delta_2(a))$ and $D_3((a,u))=(\tau_1(u),\tau_2(u))$ are derivations on $\lau$.
\par
It is clear that if $\tau_1 =0$, then by Theorem \ref{lau-der}, $\tau_2$ becomes a derivation.(Some conditions on $\U$ that force $\tau_1$ to be zero map are: $\U$ has a bounded approximate identity, $\U$ is unital and $ann_{\U}\U=\{0\}$). For instance, if $\U$ is a faithful,  a semisimple or any Banach algebra having an approximate identity, then $ann_{\U}\U=\{0\}$. By Corollary \ref{tak}(iv), the continuity of derivations on $\lau$ implies the continuity the derivations on $\U$. Particularly, if every derivation on $\A\times_{\theta}\A$ is continuous, then so is every derivation on $\A$.
\begin{prop}
Let $\A$ and $\U$ be semisimple Banach algebras. Then every derivation $D:\lau\rightarrow\lau$ is continuous.
\end{prop}
\begin{proof}
	By \cite{mn}, Theorem 3.1, $\lau$ is semisimple if and only if both $\A$ and $\U$ are semisimple. Now the result follows from Johnson's theorem. 
\end{proof}
\begin{cor}
	Suppose that  $\A$ is a $C^{*}$-algebra and $\U$ is a Banach algebra with a bounded approximate identity. Then every derivation on $\lau$ is continuous if and only if every derivation on $\U$ is continuous.
\end{cor}
\begin{proof}
	First suppose that every derivation on $\U$ is continuous. Thus for every derivation $D:\lau\rightarrow\lau$ we have
	\[D((a,u))=(\delta_1(a)+\tau_1(u),\delta_2(a)+\tau_2(u)),\]
	such that $\delta_1, \delta_2$ are derivations and by Ringrose's result \cite{rin}  are continuous.  Since $\U$ has a bounded approximate identity, $\tau_1 =0$ and  $\tau_2$ becomes a derivation. Thus $D$ is continuous if and only if $\tau_2$ is continuous. The converse is clear so omitted.  	
\end{proof}
\begin{rem}
	Let $\A$ be a commutative and $\U$ a semisimple Banach algebra with $\U$ having a bounded approximate identity. Then every derivation $D:\lau\rightarrow\lau$ is of the form $D=D_1+D_2$ where $D_1 ((a,u))=(\delta_1 (a),0)$ and $D_2 ((a,u))=(0,\tau_2(u))$ are derivations on $\lau$ such that $D_2$ is continuous. In particular, in this situation if every derivation on $\A$ is continuous, then every derivation on $\lau$ is continuous.
\end{rem}
\begin{thm}
	Let $\A$ and $\U$ be Banach algebras with bounded approximate identity. If $\A$ is commutative. Then every derivation $D:\lau\rightarrow\lau$ can be written as $D=D_1+D_2$ where $D_1 ((a,u))=(\delta_1 (a),\tau_2(u))$ and $D_2 ((a,u))=(0,\delta_2(a))$ are derivations on $\lau$. Moreover, $D_1$ is continuous if either of the following conditions holds.
	\begin{enumerate}[(i)]
		\item
		There is a surjective $\A$-module homomorphism $\phi:\A\rightarrow\U$ and $\delta_1$ is continuous.
		\item
		There is an injective $\A$-module homomorphism $\phi:\A\rightarrow\U$ and $\tau_2$ is continuous.
	\end{enumerate}
\end{thm}
\begin{proof}
	\begin{enumerate}
		\item[(i)]
		Define $\psi:\A\rightarrow \U$ by $\psi = \tau_2 \circ\phi - \phi \circ\delta_1$. It is easy to see that $\psi$ is a continuous  left $\A$-module homomorphism. Similarly,  $\phi \circ\delta_1$  is continuous and  so is $\tau_2 \circ\phi$. On the other hand, $\phi$ is surjective. So $\C(\tau_2 \circ\phi)=\C(\tau_2)=\{0\}$ by \cite{da}, Proposition $5.2.2$. Thus $\tau_2$ is continuous. 
		\item[(ii)]
		The proof is similar to part $(i)$. 
	\end{enumerate}
\end{proof}
\begin{rem}\label{pri}
	For every continuous derivation $\delta:\A\rightarrow\A$, we have $\delta(\A)\subseteq \ker \theta=ann_{\A}\U $. Since Sinclair's theorem implies that $\delta(P)\subseteq P$ for every primitive ideal $P$ of $\A$. So in this case for continuous derivation $D:\lau\rightarrow\lau$ we  have $\delta_1 (\A)\subseteq \ker \theta =ann_{\A}\U $ and  $\delta_2 (\A)\subseteq ann_{\U}\U$. Also, $D=D_1+D_2+D_{3}$ where $D_1((a,x))=(\delta_1(a),0)$, $D_{2}((a,x))=(0,\delta_2(a))$ and $D_3((a,x))=(\tau_1(x),\tau_2(x))$ are all continuous.  
\end{rem}
Note that if $\delta:\A\rightarrow\U$ is a continuous derivation, this does not necessarily imply $\delta(\A)\subseteq ann_{\U}\U$, as the following example shows.
\begin{exm}
	Assume that $G$ is a non-discrete abelian group. It has been shown in \cite{br} that there is a nonzero continuous point derivation $d$ at a nonzero character $\theta$ on $M(G)$. Now consider $M(G)\times_{\theta} \mathbb{C}$. Every derivation from $M(G)$ into $\mathbb{C}_\theta$ is a point derivation at $\theta$. It is clear that $ann_{\mathbb{C}}\mathbb{C} =\{0\}$. But $d\in Z^{1}(M(G),\mathbb{C}_\theta)$ is a nonzero derivation such that $d(M(G))\not\subseteq ann_{\mathbb{C}}\mathbb{C} =\{0\}$.
\end{exm}
\section{The multipliers}
In this section we turn our attention to the multipliers of Lau products. As before,  $\A, \U$ are Banach algebras, $\theta\in\Delta (\A)$ is a nonzero character and $\A\times_{\theta}\U$ is the associated Lau Banach algebra. 
\par 
In the following theorem we characterize the multipliers on $\A\times_{\theta}\U.$
\begin{thm}
	A linear mapping $T:\lau\rightarrow\lau$ is a multiplier if and only if there are some linear mappings $R_{1}:\A\rightarrow\A, \, R_{2}:\A\rightarrow\U, S_{1}:\U\rightarrow\A  $ and $S_{2}:\U\rightarrow\U$ with 
		\[T((a,u))=(R_{1}(a)+S_{1}(u),R_{2}(a)+S_{2}(u))\quad \quad (a\in\A,u\in\U)\]
	satisfying the following conditions:
\begin{enumerate}[(i)]\label{lau-multi}
\item 
$R_1:\A\rightarrow\A$ is a multiplier,
\item 
$aS_{1}(u)=S_{1}(u)a=0$,
\item 
$\theta(a)R_2(a')=\theta(a')R_{2}(a),$
\item 
$\theta(a)S_{2}(u)=\theta(R_{1}(a))u+R_{2}(a)u=\theta(R_{1}(a))u+u R_{2}(a),$
\item 
$\theta(S_{1}(u))u'+S_{2}(u)u'=\theta(S_{1}(u'))u+uS_{2}(u')$,
\end{enumerate}
for all $a,a'\in\A$ and $ u, u'\in\U.$	
\end{thm}
\begin{proof}
	First suppose that $T$ is a multiplier. Since $T$ is linear, there exist some linear mappings $R_{1}:\A\rightarrow\A, \, R_{2}:\A\rightarrow\U,\, S_{1}:\U\rightarrow\A$ and $S_2:\U\rightarrow\U$ with 
	\[T((a,u))=(R_{1}(a)+S_{1}(u),R_{2}(a)+S_{2}(u))\]
for all $a\in\A$ and $u\in\U$. By the definition,  $T((a,u))(a',u')=(a,u)T((a',u'))$ For all $a,a'\in\A$ and $u,u'\in\U$. If we substitute $u=u'=0$, then we deduce that $R_{1}$ is a multiplier and $\theta(a)R_2(a')=\theta(a')R_{2}(a)$ for all $a,a'\in\A$. Similarly, substituting  $a=a'=0$ yields $(v)$. If we put $a'=0, u=0$ and $a=0, u'=0$ respectively, we obtain  equalities given in $(iv)$. Also putting $a'=u'=0$, we conclude that $aS_{1}(u)=S_{1}(u)a=0$ for all $a\in\A$ and $u\in\U.$\\   
The converse is straightforward and is left for the reader.  	
\end{proof}
In view of the above theorem, in the sequel we consider any  multiplier $T:\lau\rightarrow\lau$ as
	\[T((a,u))=(R_{1}(a)+S_{1}(u),R_{2}(a)+S_{2}(u))\quad \quad (a\in\A,u\in\U)\]
	in which the mentioned maps satisfy the conditions $(i)-(v)$.
 Part $(iv)$ of the above theorem implies that $R_2$ always maps $\A$ into the center of $\U$ (that is,  $R_{2}(\A)\subseteq Z(\U)$) and also if we put $u=u'$ in $(v)$, we conclude that $S_{2}(u)u=uS_{2}(u)$ for all $u\in\U$. Moreover by $(ii)$, $S_{1}(\A)\subseteq\ann_{\A}\A$. 
 \par
 We now state the following theorem. 
\begin{thm}\label{multi-cont}
	Suppose that $T:\lau\rightarrow\lau$ is a multiplier with 
	\[T((a,u))=(R_{1}(a)+S_{1}(u),R_{2}(a)+S_{2}(u))\quad \quad (a\in\A,u\in\U).\]
Then 
\begin{enumerate}[(i)]
	\item 
The maps $R_{2}:\A\rightarrow\U$ and   $S_{2}:\U\rightarrow\U$ are automatically continuous. 
	\item 
	$\C(R_{1}), \C(S_{1})\subseteq \ker\theta$.
\end{enumerate}	
\end{thm}
\begin{proof}
	\begin{enumerate}[(i)]
		\item 
	Let $s\in\C(S_{2})$. Thus there exists some $u_{n}$ in $\U$ for which $u_{n}\to 0$ and $S_{2}(u_{n})\to s$. By Theorem \ref{lau-multi}-$(iv)$ we have 
\[\theta(a)S_{2}(u_n)=\theta(R_{1}(a))u_{n}+R_{2}(a)u_{n}\]
for all $a\in\A.$ Letting $n$ tend to infinity, we obtain $\theta(a) s=0$ for all $a\in\A$. So $s=0$ and hence $S_{2}$ is continuous. For the continuity of $R_2$, suppose that  $s'\in\C(R_{2})$ and   $a_{n}$ is a sequence in $\A$ with $a_{n}\to 0 $ and $R_{2}(a_{n})\to s'.$ By part $(ii)$ of the preceding theorem, 
\[\theta(a')R_{2}(a_{n})=\theta({a_{n}}) R_{2}(a')\] 
for all $a'\in\A$. By taking limits one obtains $s'=0$. Thus $R_{2}$ is continuous.
\item 
 Let $a_{0}\in \C(R_{1}) $. Then there exists $a_{n}\subseteq\A$ such that  $a_{n}\to 0 $ and $R_{1}(a_{n}) \to a_0.$ Therefore 
\[\theta(a_{n})S(u)=\theta(R_{1}(a_{n}))u+R_{2}(a_{n})u\]
for all $y\in\U$. So $\theta(a_0)u=0$ for all $u\in\U$. Thus $a_{0}\in\ker\theta$. The other inclusion is similar. 
\end{enumerate}
\end{proof}
If $\A$ is  faithful Banach algebra, then every multiplier on $\A$ is continuous. On the other hand, in this case $S_{1}=0$, since by Theorem \ref{lau-multi}-$(ii)$,  $S_{1}(\A)\subseteq\ann_{\A}\A =\{0\}$, so $S_{2}:\U\rightarrow\U$ is then a multiplier on $\U$. Therefore we deduce the following result. 
\begin{thm}\label{faith-cont}
Suppose that $\A$ and $\U$ are Banach algebras such that $\A$ is faithful. Then every multiplier on $\lau$ is continuous. 
\end{thm}

It is well-known  in each of the following cases the Banach algebra $\A$ is faithful : 
\begin{enumerate}
	\item 
	$\A$ is unital.
	\item 
	 $\A$ has an approximate identity(for example, $\A$ is a $C^{*}-$algebra).
	\item 
	$\A$ is semiprime.
	\item 
	$\A$ is a semisimple.
\end{enumerate} 

Suppose that $\A,\U$ are Banach algebras such that $\A$ is a faithful. Then an easy calculation shows that
 \[\ann_{(\lau)}(\lau)=\{(0,u):u\in \ann_{\U}\U\}\cong \ann_{\U}\U.\]
Therefore if $\A$ is faithful, $\lau$ is faithful if and only if $\U$ is so. Therefore if we assume that $\U$ is not faithful, then $\lau$ is not faithful as well and by Theorem \ref{faith-cont}, all multipliers $T:\lau\rightarrow\lau$ are continuous. This result can be interesting by itself since it can provide non-faithful Banach algebras on which every multiplier is continuous.
 \bibliographystyle{plain}

\begin{thebibliography}{20}
	\bibitem{ba4}
	W. G. Bade and P. C. Curtis, \textit{The continuity of derivation of Banach algebras}, J. Functional analysis \textbf{16} (1974), 372--378.
	
	\bibitem{ba3} W. G. Bade and P. C. Curtis, \textit{the structure of module derivations of Banach algebras of differentiable functions}. J. of Fun.\textbf{28} (1978), 226-247.
	
	\bibitem{bir} F. T. Birtal, \textit{Isomorphism and isometric multipliers}, Proc. Amer. Math. Soc. \textbf{13} (1962), 204-210.
	
	\bibitem{br} G. Brown and W. Moran, \textit{Point derivations on $M(G)$}, Bull. London Math. Soc. \textbf{8} (1976), 57--64.
	
	\bibitem{chr} E. Christensen, \textit{Derivations of nest algebras}, Math. Ann. \textbf{229} (1977), 155--161.
	
	\bibitem{da} H. G. Dales, \textit{Banach Algebras and Automatic Continuity}. In: London Math. Soc. Monographs. Oxford Univ. Press, Oxford 2000.
	
	\bibitem{do} A. Donsig, B. Forrest and M. Marcoux, \textit{On derivations of seminest algebras}, Houston J. Math. \textbf{22} (1996), 375--398.
	
	\bibitem{ram} R. V.  Garimella, \textit{Continuity of derivations on some semiprime Banach algebra},	Proc. Amer. Math. Soc. \textbf{99} (1987), 289-292.
	
	\bibitem{gr} N. Gr\o nb\ae k, \textit{Commutative Banach algebras, module derivations and semigroups}, J. London Math. Soc. \textbf{40}(2)  (1989), 137--157.
	
	
	\bibitem{hel} S. Helgason, \textit{Multipliers of Banach algebras}, Ann. of Math, \textbf{64} (1956), 240-254.

	\bibitem{john2} B. E. Johnson, \textit{Cohomology in Banach Algebras}, Mem. Amer. Math. Soc. \textbf{127} (1972).
	
	
	\bibitem{john1} B. E. Johnson and A. M. Sinclair, \textit{Continuity of derivations and a problem of Kaplansky}, Amer. J. of Math. \textbf{90} (1968), 1067--1073.
	
	
		\bibitem{ka} R. V. Kadison, \textit{Derivations of operator algebras}, Ann. Math. \textbf{83} (1966), 280--293.
	
	\bibitem{lar} R. Larsen, \textit{An Introduction to the Theory of Multipliers}, Springer-Verlag, New York, 1971.
	
	\bibitem{lau} A. T. Lau, \textit{Analysis on a class of Banach algebras with application to harmonic analysis on locally compact groups and semigroups}, Fund. Math. \textbf{118} (1983), 161--175.
	
	\bibitem{laur} K. B. Laursen, and M. M. Neumann, \textit{An Introduction to Local Spectral Theory}, Oxford University Press, New York, 2000.
	
	\bibitem{loy3} R. Loy, \textit{Continuity of derivations on topological algebras of power series}, Bull. Aust. Math. Soc. \textbf{1}(3) (1969), 419-424.
	
	\bibitem{loy2}	R. J. Loy and G. A. Willis, \textit{Continuity of derivations on $B(E)$ for certain Banach spaces $E$}, J. London Math. Soc. s2-40 (1989), 327-346. 
	
	\bibitem{mn} M. S. Monfared, \textit{On certain products of Banach algebras with applications to harmonic analysis}, Studia Math. \textbf{178} (2007), 227--294.
	
	\bibitem{pera} A. M. Peralta and B. Russo, \textit{Automatic continuity of derivations on $C^{*}$-algebras and JB$^{*}$-triples}, J. Algebra. \textbf{399}, 960--977.
	
	\bibitem{rin} J. R. Ringrose, \textit{Automatic continuity of derivations of operator algebras}, J. London Math. Soc. \textbf{5} (1972), 432--438.
	
    \bibitem{rund}	V. Runde, \textit{ Automatic continuity of derivations and epimorphisms}, Pacific J. Math. {\bf 147}(2) (1991), 365--374.
	
	\bibitem{sa} S. Sakai, \textit{Derivations of $W^{*}$-algebras}, Ann. Math. \textbf{83} (1966), 273--279.
	
	\bibitem{sinc} A. M. Sinclair, \textit{Continuous derivations on Banach algebras}, Proc. Amer. Math. Soc. \textbf{20} (1969), 166--170.
	
	\bibitem{sin} I. M. Singer and J. Wermer, \textit{Derivations on commutative normed algebras}, Math. Ann. \textbf{129} (1955), 260--264.
	
	\bibitem{tho} M. P. Thomas, \textit{The image of a derivation is contained in the radical}, Ann. of Math. \textbf{128} (1988), 435--460.
	
	\bibitem{vi3} A. R. Villena, \textit{Automatic continuity in associative and nonassociative context}, Irish Math. Soc.\textbf{ Bulletin 46} (2001), 43--76.

	\bibitem{vil4} A. R. Villena, \textit{Continuity of derivations on $H^{*}$-algebras}, Proc. Amer. Math. Soc. \textbf{122} (1994), 821-826.
	
	\bibitem{vi1} A. R. Villena, \textit{Derivations with a hereditary domain}, J. London Math. Soc. \textbf{57} (1998), 469--477. 
	
	\bibitem{vi2} A. R. Villena, \textit{Derivations with a hereditary domain, II}, Studia Math. \textbf{130} (1998), 275--291.
	
	
	\bibitem{wan}J. K. Wang, \textit{Multipliers of commutative Banach algebras}, Pacific J. Math, \textbf{11}(1961), 1131--1149.
	
\end{thebibliography}

\end{document}